\documentclass[a4paper,12pt]{article}

\usepackage{amssymb,bm,mathrsfs,bbm,amscd}
\usepackage{dsfont}
\usepackage{filecontents}
\usepackage[colorlinks=true, citecolor=blue, linkcolor=red, urlcolor=blue]{hyperref}
\usepackage[tbtags]{amsmath,nccmath}
\usepackage{xcolor}

\makeatletter
\newcommand*{\Xbar}{}%
\DeclareRobustCommand*{\Xbar}{%
  \mathpalette\@Xbar{}%
}
\newcommand*{\@Xbar}[2]{%
  \sbox0{$#1\mathrm{X}\m@th$}%
  \sbox2{$#1X\m@th$}%
  \rlap{%
    \hbox to\wd2{%
      \hfill
      $\overline{%
        \vrule width 0pt height\ht0 %
        \kern\wd0 %
      }$%
    }%
  }%
  \copy2 %
}
\makeatother
\usepackage{lastpage,graphicx}
\usepackage [top=1.5in , bottom=1.5in, left=1.3in, right=1.2in] {geometry}

\usepackage{enumitem}
\usepackage{relsize}
\usepackage{tikz}
\usepackage{pgfplots}\pgfplotsset{compat=newest}
\usepackage{caption}
\usepackage[utf8]{inputenc}
\usepackage{lastpage}
\usepackage{titlesec}
   \titleformat*{\section}{\large\bfseries}
   \titleformat*{\subsection}{\bfseries}
\usepackage[english]{babel}
\usepackage{fancyhdr}
\usepackage{bigints}

\usepackage{natbib}
\usepackage{amsthm}

\theoremstyle{plain}
\newtheorem{thm}{Theorem}[section]

\theoremstyle{definition}
\newtheorem{defn}{Definition}[section]

\theoremstyle{remark}

\numberwithin{equation}{section}

\def\({\left(}
\def\){\right)}
\def\[{\left[}
\def\]{\right]}
\newtheoremstyle{theorem}{\topsep}{\topsep}%
     {}
     {}
     {}
     {}
     {.5em}
     {\thmname{{\bfseries #1}}\thmnumber{ #2}\thmnote{ #3}.}

\theoremstyle{theorem}

\newtheoremstyle{remark}{\topsep}{\topsep}%
     {}
     {}
     {\it}
     {}
     {.5em}
     {\thmname{{ #1}}\thmnumber{ #2}\thmnote{ #3}.}
\theoremstyle{remark}

\usepackage{amssymb}

\title{fCIR}
\author{{ Mpanda}\;\footnote{Corresponding author,
\:\:email: mpandmm@unisa.ac.za}\;\;\; and \;\;\;{MM. Mpanda}\\
{\small  UNISA} \\
}
\date{}

\begin{document}

\thispagestyle{empty}
   ~~
\vspace{1cm}
\begin{center}

\textbf{\Large{Generalisation of Fractional Cox-Ingersoll-Ross Process}} \\[2mm]
            \vspace{0.5cm}
            \large{Marc M. Mpanda\footnote{Corresponding author, mpandmm@unisa.ac.za}, Safari Mukeru*, Mmboniseni Mulaudzi**} \\[5mm]
                                  Department of Decision Sciences\\
             \small{University of South Africa, P. O. Box 392, Pretoria, 0003. South Africa}\\
             \small{mpandmm@unisa.ac.za, \, *mukers@unisa.ac.za, \, **Mulaump@unisa.ac.za}\\
\end{center}

\noindent

 \begin{abstract}

\noindent
In this paper, we define a generalised fractional Cox-Ingersoll-Ross process  $(X_t)_{t\geq 0}$ as a square of singular stochastic differential equation with respect to fractional Brownian motion with Hurst parameter $H \in (0,1)$ taking the form $dZ_t = \left(f(t, Z_t) Z_t^{-1}dt + \sigma dW_t^H\right)/2$, where $f(t,z)$ is a continuous function on $\mathds{R}_+^2$. Firstly, we show that this differential equation has a unique solution $(Z_t)_{t\geq 0}$ which is continuous and positive up to the time of the first visit to zero. In addition,  we prove that the stochastic process $(X_t)_{t\geq 0}$ satisfies the differential equation $dX_t = f(t,\sqrt{X_t})dt + \sigma\sqrt{X_t}\circ dW_t^H $ where $\circ$ refers to the Stratonovich integral. Moreover, we prove that the process $(X_t)$ is strictly positive everywhere almost surely for $H>1/2$. In the case where $H<1/2$, we consider a sequence of increasing functions $(f_n)$ and we prove that the probability of hitting zero tends to zero as $n \to \infty$. These results are illustrated with some simulations using the generalisation of the extended Cox-Ingersoll-Ross process.\\

\vspace{0.5cm}

\end{abstract}

\newpage
\section{Introduction}\label{Intro}

In mathematics of finance, the Cox-Ingersoll-Ross (\emph{CIR}) process is a diffusion process that was initially introduced by \citet{cox1985theory} to model the dynamics of interest rates. In a probability space $(\Omega, \mathcal{F}, \mathds{P})$, the \emph{CIR} process satisfies the following stochastic differential equation:
\begin{equation}\label{1-1}
dX_t = \theta(\mu-X_t)dt + \sigma\sqrt{X_t}dW_t,
\end{equation}
where $\theta$  is a positive parameter that represents the speed of reversion of the stochastic process $(X_t)_{t\geq 0}$ towards its long-run mean $\mu >0$, $\sigma>0$ is the volatility of $(X_t)_{t\geq 0}$ and  $(W_t)_{t\geq 0}$ is the standard Brownian motion.\\

\noindent
The \emph{CIR} process has several interesting properties: its sample paths are strictly positive provided that the condition $2\theta \mu > \sigma^2$ holds, it is mean reverting in the sense that the process is pulled towards its long-run mean $\mu$ when it goes higher or lower than $\mu$. Moreover, the \emph{CIR} process admits a stationary distribution and it is ergodic. For more details, see e.g. \citet{going2003survey}, \citet{chou2006some} and \citet{guo2008note} with references therein. These properties were the main motivations of using the \emph{CIR} process in modeling the dynamics of interest rates \citep{cox1985theory} and the random behavior of spot volatility \citep{heston1993closed}.\\

\noindent
Since the standard \emph{CIR} process is driven by a Brownian motion, it does not display memory. Recently, it was shown that there is a certain range of dependency within financial data. For example, spot volatilities may display long-range dependency as discussed by \citet{comte1998long} and \citet{chronopoulou2010hurst}, or short range dependency known as ``\emph{rough volatility}'' as demonstrated by \citet{gatheral2018volatility} and \citet{livieri2018rough} with references therein. This was a motivation of replacing the standard Brownian motion in (\ref{1-1}) by a fractional Brownian motion (\emph{fBm}) as source of randomness.\\

\noindent
Although the empirical definition of fractional Cox-Ingersoll-Ross (\emph{fCIR}) process can be now formulated as a \emph{CIR} process driven by a \emph{fBm}, \citet{mishura2018stochastic} define the \emph{fCIR} process  as a stochastic process $(X_t)_{t\geq 0}$ given by
$$
X_t(\omega) = Z^2_t(\omega)\mathbf{1}_{[0,\tau(\omega))}, ~~~~~~\forall t \geq 0, ~~\omega\in\Omega,
$$
\noindent
where $(Z_t)_{t\geq 0}$ is a fractional Ornstein–Uhlenbeck process that satisfies the stochastic differential equation $dZ_t = -\frac{1}{2}\theta Z_t dt + \frac{1}{2}\sigma dW_t^H$, with $W_t^H$ a \emph{fBm} of Hurst parameter $H\in(0,1)$, and where $\tau$ is the first time the process $(Z_t)_{t\geq 0}$ hits zero. On the other hand, \citet{mishura2018fractional} consider the process $(Z_t)_{t\geq 0}$ defined by the equation
 \begin{equation}\label{1-2}
            dZ_t = \mfrac{1}{2}\Big(\mu -  \theta Z_t^2 \Big)Z_t^{-1}dt + \mfrac{1}{2}\sigma dW_t^H,
        \end{equation}
\noindent
(where $\mu$, $\theta$  and  $\sigma > 0$ are  parameters) and the corresponding process $$
X_t(\omega) = Z^2_t(\omega)\mathbf{1}_{[0,\tau(\omega))}, ~~~~~~\forall t \geq 0, ~~\omega\in\Omega.
$$

\noindent
\citet{mishura2018fractional} proved that the \emph{fCIR} process $(X_t)_{t\geq 0}$ given by (\ref{1-2}) verifies the equation given by

   \begin{eqnarray} \label{mishstrat1}
X_t = X_0 + \int_{0}^{t} (\mu - \theta X_s)ds + \sigma\int_{0}^{t}\sqrt{X_s}\circ dW_s^H,
\end{eqnarray}

\noindent
where $\int_{0}^{t}\sqrt{X_s}\circ dW_s^H$ is a Stratonovich integral with respect to the \emph{fBm} $(W_s^H)_{s\geq 0},$\\$H\in(0,1)$. Moreover, they proved that the process $(X_t)_{t\geq 0}$ is strictly positive and will never hit zero for $H > \frac{1}{2}$. For $H<\frac{1}{2}$, they obtained that the probability of hitting zero converges to $0$ when the speed of reversion $\theta$ tends to infinity. In addition, \citet{hong2019optimal} investigated  strong convergence of some numerical approximations for \emph{fCIR} process in the case where $H\geq \frac{1}{2}$.\\

\noindent
In this paper, we extend the results of \citet{mishura2018fractional} to a general process $(Z_t)_{t\geq 0}$ defined by the stochastic differential equation
           \begin{equation}\label{1-3}
            dZ_t = \mfrac{1}{2}f(t,Z_t)Z_t^{-1}dt + \mfrac{1}{2}\sigma dW_t^H,\,\, Z_0 >0,
        \end{equation}
where $f:[0, \infty) \times [0, \infty) \to (-\infty, \infty)$, $(t, x) \mapsto  f(t,x)$, is a continuous drift function. We consider as previously a \emph{fCIR} process defined by
$$
X_t(\omega) = Z^2_t(\omega)\mathbf{1}_{[0,\tau(\omega))}, ~~~~~~\forall t \geq 0, ~~\omega\in\Omega.
$$

\noindent

\noindent
This general case has been previously investigated by \citet{hu2008singular} and \citet{nualart2002regularization} under some additional assumptions on the drift function $f$.
\citet{hu2008singular} proved that if $H > 1/2$ and if the drift function $f(t, x)$ is such that the function $g$ given by $g(t,x) = f(t, x)/(2x)$ satisfies the following conditions,
    \begin{itemize}
    \item[(C1)] $g: [0, \infty) \times (0, \infty) \to [0, \infty)$ is a nonnegative continuous function which has a continuous partial derivative $\partial g(t, x)/\partial x \leq 0$ for all $(t,x) \in (0, \infty) \times (0, \infty),$
    \item[(C2)] There exist $x_1> 0$, $a > \frac{1}{H} - 1$ and a continuous function $\varphi:[0, \infty) \to [0, \infty)$ with $\varphi(t) > 0 $ for all $t >0$ such that  $g(t, x) \geq \varphi(t) x^{-a}$ for all $t \geq 0$ and $0 < x < x_1$,
    \end{itemize}
then (\ref{1-3}) has a strictly positive solution $(Z_t)_{t \geq 0}$ that is, almost surely $Z_t >0$ for all $t>0$. (See Theorem 2.1 and Theorem 3.1 in \citet{hu2008singular}). In addition, they also showed that
     \begin{itemize}
       \item[(C3)] if there exists a function $h:[0, \infty) \to [0, \infty)$ which is nonnegative and  locally bounded such that  $g(t, x) \leq h(t) (1 + 1/x)$ for all $t \geq 0$ and $x >0$, then the solution $(Z_t)_{t \geq 0}$ is such that for any fixed $T>0$,
          $$\mathbb{E}\left(\sup_{0 \leq t \leq T} |Z_t|^p \right) < \infty, ~~~\forall p >0.$$
     \end{itemize}

\noindent
Our objective is to study the solution to the stochastic differential equation (\ref{1-3}) under some mild conditions weaker than conditions than (C1) and (C2). We shall consider the following conditions:
 \begin{itemize}
 \item[(D1)] The function $g: [0, \infty) \times (0, \infty) \to (-\infty,  \infty)$ defined by
$g(t,x) = f(t,x)/(2x)$ is continuous and admits a continuous partial derivative with respect to $x$ on $(0, \infty)$. In addition, there exists a number $x^*>0$ such that for every $x>x^*$, $g(t,x)<0$, for all $t\geq 0$.
  \item[(D2)] for any $T >0$, there exists  $x_{T} >0$  such that
 $$f(t, x) > 0 \mbox{  for all }  0 < t \leq T \mbox{ and } 0 \leq x \leq x_{T}.$$
  \end{itemize}

\noindent
Condition (D2) implies that for all $S>0$ and $T>0$, there exists $x_T>0$ such that  $\inf\{f(t,x): S \leq t \leq T, 0 \leq x \leq x_T\} > 0.$ \\ 

\noindent
We shall first show that condition (D1) and the initial condition $Z_0 >0$ guarantee the existence, uniqueness, continuity and positiveness of a solution $(Z_t)$ to equation (\ref{1-3}) up to the first time it hits zero. In addition, We will show that the square stochastic process $(X_t)_{t\geq 0}$ (which is also defined up to the first time  it hits zero) satisfies the stochastic differential equation
    $$
        dX_t = f(t,\sqrt{X_t})dt + \sigma\sqrt{X_t}\circ dW_t^H,\,\, X_0 > 0, H\in(0,1).
    $$

\noindent
We shall also prove that in the case where $H > 1/2$, the solution to the stochastic differential equation (\ref{1-3}) is not only positive up to the time of the first visit to zero but it is strictly positive everywhere. In other words, almost surely it never hits zero on the whole line $[0, \infty)$. It is remarkable that this result is true under mild conditions (D1) and (D2). \\

\noindent
In the case where $H <\frac{1}{2}$, we obtain that the probability of the process $(X_t)_{t\geq 0}$ hitting zero is small if the drift function $f$ is sufficiently large. More precisely, if $(f_n)_{n\in \mathds{N}}$ is an increasing sequence of continuous functions $f_n$ defined on $[0, \infty) \times [0, \infty)$ and taking values in $\mathbb{R}$ and satisfying conditions (D1) and (D2), such that  $\lim_{n\to \infty} f_n = \infty$ and $(X_t^n)$ is the solution to equation (\ref{1-3})  corresponding to $f_n$ (up to the first time it hits zero), then the probability of $(X_t^n)$ hitting zero converges to 0 as $n\to \infty$.
Our results generalize the results recently by  \citet{mishura2018fractional} for the function $f(t,x) = \frac{1}{2}(\mu - \theta x^2)$ for constants $\mu>0$ and $\theta >0$.
We provide some illustrating examples using simulation. \\

\noindent
More recently  \citet{kubilius2020estimation} studied the stochastic differential equation
\begin{eqnarray} \label{Eq1-5}
    dX_t = g(X_t) dt + \sigma X_t^\beta dW^H_t\end{eqnarray}
for $1/2 < H < 1$, $1/2 \leq \beta <  1$ and where the function $g$ is such that there exists a continuously differentiable function $f$ defined on $(0, \infty)$ such that: (1) $g(x) = x^{\beta} f(x^{1-\beta})$, (2) there exist $a>0$ and $\alpha \geq 0$ such that $f(x) > a x^{-(1+\alpha)}$ for sufficiently small $x$ and (3) there exists $K \in \mathbb{R}$ such that $f'(x) \leq K.$ Under these conditions, it is proven that equation (\ref{Eq1-5}) has a unique and positive solution  and derived an important estimator of the $H$ for the solution. In some sense our model (\ref{1-3}) extends (\ref{Eq1-5}).
It would be interesting to carry out an analysis of the $H$ parameter of the solution to equation (\ref{1-3}) as in \citet{kubilius2020estimation}. \\

\noindent
The rest of this paper is organised as follows. Section 2 discusses the existence and uniqueness of the generalised \emph{fCIR} processes.  In Section 3 we show such processes satisfy a stochastic differential equation with respect to Stratonovich integral. The  positiveness of these processes for $H>1/2$ is given in section 4 and for $H < 1/2$ in section 5. Section 6 contains some illustrations of the main results using simulation and finally the last section contains some concluding remarks.

\section{The generalised \emph{fCIR} processes}

In this section, we consider a more general process $(Z_t)_{t\geq 0}$ defined by the differential stochastic equation:
\begin{equation}\label{3-1}
        dZ_t = \frac{f(t,Z_t)}{2Z_t}dt+\frac{\sigma}{2}dW_t^H,\,\,\, Z_0 > 0
\end{equation}
where $f:[0, \infty)\times[0, \infty) \to (-\infty, \infty)$, $(t, z) \mapsto f(t,z)$ is a continuous function satisfying conditions (D1) and (D2).
We shall first discuss the existence and uniqueness of the solution to (\ref{3-1}).

\begin{thm}
If the drift function $f(t,x)$ satisfies condition $(D1)$, then for all $0 < H < 1$, equation (\ref{3-1}) has a unique solution $(Z_t)_{t\geq 0}$ which is continuous and positive up to the time of the first visit to $0$.
\end{thm}

\begin{proof}
	Let $\ell>0$ be a small number such that $\ell<Z_0$ and \textcolor[rgb]{0.50,0.25,0.25}{$\ell<x^*$}. For fixed $T>0$, consider the sequence of processes $(Z_n(t))$ defined on $[0, T]$ by
		$$Z_0(t) = Z_0, ~~~\text{for all} ~~~ t\in[0,T]$$ and for all $n\in \mathbb{N}$,
	\begin{eqnarray*}
		  Z_{n+1}(t) = \left\{ \begin{array}{ll} Z_0 + \int_0^{t} g(s, Z_n(s)) ds + \frac{\sigma}{2} W_t^H,& ~~~ \text{if} ~~~ t\leq \tau_{n,\ell}\\[2mm]
		  	\ell & ~~ \text{otherwise} \end{array}\right.
	\end{eqnarray*}
	where $g(t,x) = f(t,x)/(2x)$ and $\tau_{n,\ell} = \inf\{0 \leq t \leq T : Z_{n}(t) = \ell\}$ (the first time the process $(Z_n(t))$ reaches the level $\ell$ with $\inf(\emptyset) = +\infty$). Clearly, if $Z_n(t)$ does not reach the level $\ell$ on $[0,T]$, then
	 $Z_{n+1}(t)$ is defined  by
	 $$
	 Z_{n+1}(t) = Z_0 + \int_0^{t} g(s, Z_n(s)) ds + \frac{\sigma}{2} W_t^H
	 $$
	for all $t \in [0, T]$. For instance
	 	$$
	 	Z_{1}(t) = Z_0 + \int_0^{t} g(s, Z_0) ds + \frac{\sigma}{2} W_t^H, ~~t \in [0, T].
	 	$$

\noindent	
We want to  show that there exists a number $\eta >0 $ independent of $n$ and such that $\tau_{n,\ell} \geq \eta$ for all $n$.
	It is clear that $\tau_{n,\ell} \geq \tau_{n+1, \ell}$ because
		 $Z_{n+1}(t)  = \ell$ for all $t \geq \tau_{n,\ell}$. The function $t \mapsto g(t, Z_n(t))$ is bounded on $t \in [0, \tau_{n,\ell}]$. Indeed, for every $t \in [0, \tau_{n,\ell}]$,
		 write
		    $[0, t]  = I_1 \cup  I_2$ where
		    $I_1$ is the union of sub-intervals of $[0, t]$ where $Z_n \leq x^*$ and $I_2$ is the union of sub-intervals of $[0, t]$ where $Z_n > x^*$.
		    Then
		 $$\int_0^t g(s, Z_n(s)) ds = \int_{I_1} g(s, Z_n(s))ds + \int_{I_2} g(s, Z_n(s))ds \leq \int_{I_1} g(s, Z_n(s))ds$$
	because $g(s, Z_n(s)) < 0$  for $s\in I_2$ by condition (D1). Therefore
	\begin{eqnarray*} 
		Z_{n+1}(t)  = Z_0 + \int_0^t g(s, Z_n(s))ds + \frac{\sigma}{2} W_t^H \\
		            \leq Z_0 + \int_{I_1} g(s, Z_n(s))ds + \frac{\sigma}{2} W_t^H. \nonumber
		            	\end{eqnarray*}
	           Let
	 $$  A = \sup ( \{|g(s, z)|: s \in [0, T] \text{ and } z \in [\ell, x^*]\}).$$ Clearly $A < \infty$ because $g$ is continuous on $[0, +\infty) \times (0, +\infty).$ Because for $s \in I_1$, $Z_n(s) < x^*$, then
	 $$Z_{n+1}(t) \leq Z_0 + A t + \frac{\sigma}{2} W_t^H  \leq B $$
	 where $$B = Z_0 + A T + \frac{\sigma}{2}\sup_{0\leq t \leq T}|W_t^H|.$$
	 (Here the bound $B$ is independent of $n$). Therefore for all $t \in [0, \tau_{n,\ell}]$, we have that    $Z_{n+1}(t) \in  [\ell, B]$ for all $n \in \mathbb{N}$. Since  $\tau_{n+1, \ell} \leq \tau_{n,\ell}$, it follows in particular that
	    $$Z_{n+1}(t)  \in  [\ell, B] \mbox{ for all } 0 \leq t\leq \tau_{n+1, \ell}.$$
Moreover, since by definition,  $$Z_{n+1}(t) = Z_0 + \int_0^{t} g(s, Z_n(s)) ds + \frac{\sigma}{2} W_t^H, $$
taking $t = \tau_{n+1, \ell}$ yields
   $$ \ell  = Z_0 + \int_0^{\tau_{n+1,\ell}} g(s, Z_n(s)) ds + \frac{\sigma}{2} W_{\tau_{n+1, \ell}}^H.$$
   Set $$K = \sup ( \{|g(s, z)|: s \in [0, T] \text{ and } z \in [\ell,  B]\}),$$then
   $$ \ell \geq Z_0 - K \tau_{n+1, \ell} +  \frac{\sigma}{2} W_{\tau_{n+1, \ell}}^H.$$
That is
      $$\frac{\sigma}{2} W_{\tau_{n+1, \ell}}^H \leq  \ell -  Z_0 + K \tau_{n+1,\ell}.$$	
This implies that
 $$\tau_{n+1,\ell}\geq \inf\{ t \geq 0: \frac{\sigma}{2} W_t^H \leq  \ell -  Z_0 + K t\} .$$
Set $$\eta = \inf\{ t \geq 0: \frac{\sigma}{2} W_t^H \leq  \ell -  Z_0 + K t\}.  $$

\noindent
Clearly $\eta > 0$ because obviously the fractional Brownian motion $(W_t^H)$  starts at 0, that is, $W_0^H = 0$ and $\ell < Z_0$. Hence, $\tau_{n+1,\ell}\geq \eta>0$ uniformly for $n$ (and $\eta$ is independent of $n$). \\

\noindent
Let $\tau_\ell = \inf_{n\geq 0} \tau_{n,\ell}$, then $\tau_\ell \geq \eta>0$. We will then show that the problem has a positive solution on the interval $[0,\tau_\ell]$. For all $n$ and all $t\in[0,\tau_\ell]$, $Z_n(t)\geq \ell$ and $Z_n(t)\leq B$. \\

\noindent
Since the function $g(t,x)$ admits a partial derivative with respect to $x$ on $(0, \infty)$, then in particular for fixed $t$, the function $(t,x) \mapsto g(t,x)$ is uniformly Lipschitz for $x$ in a bounded closed interval away from 0.\\

\noindent
Since for all $t\in [0, \tau_{\ell}]$, $Z_n(t)\in [\ell, B]$, then there exists $C>0$ such that $$|(g(t,Z_{n}(t)) - g(t,Z_{n-1}(t)))| \leq C |Z_n(t) - Z_{n-1}(t)|$$ for all $t\in [0,\tau_\ell]$. Therefore,
~~
\begin{eqnarray*}
	|Z_{n+1}(t) - Z_n(t)|  &\leq&  \int_{0}^{t} |(g(s,Z_{n}(s)) - g(s,Z_{n-1}(s)))| ds \\
	                   &\leq & C  \int_0^{t} |Z_n(s) - Z_{n-1}(s)| ds.
\end{eqnarray*}

\noindent	
Then an application of Gr\"onwall's lemma implies that the sequence $(Z_n(t))$ converges uniformly on the interval $[0, \tau_{\ell}]$ and hence its limit is a positive continuous solution to (\ref{1-3}) on $[0, \tau_{\ell}]$. Therefore, equation (\ref{1-3}) admits a positive solution up to the first time it hits the level $\ell$. For the uniqueness of the solution, if $(Z_t)$ and $(Y_t)$ are two solutions on some interval $[0, \tau_\ell)$ starting at the same point $Z_0$, then
for any $t < \tau_\ell$,
\begin{eqnarray*}
	|Z_t - Y_t| \leq  \int_0^{t} |(g(s, Z_s) - g(s, Y_s))| ds  \leq  C \int_0^{t} |Z_s - Y_s| ds.
\end{eqnarray*}

\noindent
Again Gr\"onwall's lemma implies that $Z_t = Y_t$ everywhere in $[0, \tau_\ell).$ Since $\ell>0$ can be taken arbitrary small, this implies the existence of a solution up to the first time it hits $0$.
\end{proof}

\begin{defn}\label{Def3-1}
The stochastic process $(X_t)_{t\geq 0}$ defined by
           \begin{equation}\label{3-3}
            X_t = Z^{2}_t\mathbf{1}_{[0,\tau)}(t),\,\, t\geq 0,\,\,\tau = \inf\{t>0: Z_t = 0\}
        \end{equation}
\noindent
where $(Z_t)_{t\geq 0}$ is the solution to (\ref{3-1}) will be called the generalised \emph{fCIR} process defined by the function $f$.\\[-3mm]
\end{defn}

\noindent
\textbf{Remark.} When $f(t,z) = (\mu - \theta z^2)$ where $\theta$ and $\mu$ are constants, the generalised \emph{fCIR} process $(X_t)_{t\geq 0}$ coincides with the \emph{fCIR} process given by \citet{mishura2018fractional}.
In addition, when the speed of reversion or the long-run mean are time dependent, that is $\theta = \theta_t$ or $\mu = \mu_t$ with $f(t,z) = \theta_t(\mu_t - z^2)$, the process $(X_t)_{t\geq 0}$ can be regarded as an extended \emph{fCIR} process (or a fractional Hull-White model that has been used by \citet{pan2017pricing} for pricing options). The latter process  is very important not only because of the mean-reverting and positiveness properties but also because of  the possibility of a perfect calibration of parameters to the market data.

\section{Connection to Stratonovich integral}
We recall that given two stochastic processes $(X_t)_{t\in [0,T]}$ and $(Y_t)_{t \in [0,T]}$, the pathwise Stratonovich integral $\int_0^T Y_s\circ dX_s$ is defined as a pathwise limit (when it exists) given by
\begin{equation}\label{2-2}
 \lim_{n\rightarrow \infty} \sum_{i=1}^n \left(\frac{Y_{t_i} +Y_{t_{i-1}}}{2}\right)(X_{t_i} - X_{t_{i-1}}),
\end{equation}
where $0 = t_0 < t_1 < \ldots< t_{n-1} < t_n = T$ is a partition of the interval $[0, T]$ such that $\sup_{0 \leq i \leq n}|t_{i} - t_{i-1}| \to 0$ as $n \to \infty$.
We have the following result.
\begin{thm}\label{Prop3-2}
Assume that the function $f: [0, \infty) \times [0, \infty) \to \mathbb{R}$ is continuous and satisfies $(D1)$.
Then the corresponding generalised fCIR process $(X_t)$  defined by $(\ref{3-3})$ up to the first time it hits zero satisfies the equation:
\begin{equation}\label{3-4a}
X_t = X_0 + \int_{0}^{t}f(s, \sqrt{X_s})ds + \sigma \int_{0}^{t}\sqrt{X_s}\circ dW_s^H,
\end{equation}
\noindent
where $\int_{0}^{t}\sqrt{X_s}\circ dW_s^H$ is the Stratonovich integral. 
\end{thm}

\noindent
\textbf{Proof.} Our proof is a generalisation of a proof given by \citet{mishura2018fractional} applied to the particular function $f(t, x) =  f(t,x) = \frac{1}{2}(\mu - \theta x^2).$  As already discussed condition (D1) implies  the uniqueness of the solution $(Z_t)$ up to the first time it hits zero.
For $\tau:= \inf\{s>0 : Z_s = 0\}$ and $t\in[0,\tau)$ fixed, we have from (\ref{3-1}) and (\ref{3-3}) that
        \begin{equation}\label{2-10}
            X_t = Z_t^{2} = \(Z_0 + \frac{1}{2}\int_0^t f(s,Z_s)Z_s^{-1}ds + \frac{\sigma}{2}dW_t^H\)^{2},
        \end{equation}
where $Z_0$ is an initial value of the stochastic process $(Z_t)_{t\in[0,\tau)}$. In discrete time, assume that the interval $[0,t]$ is subdivided into $N$ equal subintervals with\\ $0<t_1<\cdot\cdot\cdot<t_N=t$, the time-steps $\delta t = t/N$, and $t_i = i\delta t, ~  i = 0, \cdot\cdot\cdot, N.$ Then it follows that

\begin{equation*}
\begin{aligned}
X_t &= X_0 + \sum_{i = 1}^N (X_{t_i}-X_{t_{i-1}}) \\
    &= X_0 + \sum_{i = 1}^N \(\[Z_0 + \int_0^{t_i}\frac{1}{2}f(s,Z_s)Z_s^{-1}ds + \frac{\sigma}{2}dW_{t_i}^H\]^{2}\right.\\
    & ~~~~~~~~ \left.- \[Z_0 + \frac{1}{2}\int_0^{t_{i-1}}f(s,Z_s)Z_s^{-1}ds + \frac{\sigma}{2} W_{t_{i-1}}^H\]^{2}\).
\end{aligned}
\end{equation*}
Then
\begin{equation*}
\begin{aligned}
X_t & =  X_0 + \sum_{i = 1}^N \[\frac{1}{2}\int_{t_{i-1}}^{t_i}f(s,Z_s)Z_s^{-1}ds + \frac{\sigma}{2}\( W_{t_i}^H - W_{t_{i-1}}^H\)\]\\
  & ~~~~~~~~ ~~~~~~~~\times \[2 Z_0 + \frac{1}{2}\(\int_0^{t_{i}}f(s,Z_s)Z_s^{-1}ds + \int_0^{t_{i-1}}f(s,Z_s)Z_s^{-1}ds\) \right. \\
  & ~~~~~~~~ ~~~~~~~~ ~~~~~~~~  \left. + \frac{\sigma}{2}\( W_{t_i}^H + W_{t_{i-1}}^H\)\].
  \end{aligned}
\end{equation*}
The last equation above is obtained by factorising the difference of two squares. After some expansions, we obtain that

\begin{equation*}
\begin{aligned}
  X_t & =  X_0   +  Z_0 \sum_{i = 1}^N \int_{t_{i-1}}^{t_i}f(s,Z_s)Z_s^{-1}ds \\
  &  ~~~~~~~~  +  \frac{1}{4} \sum_{i = 1}^N \int_{t_{i-1}}^{t_i}f(s,Z_s)Z_s^{-1}ds\(\int_0^{t_{i}}f(s,Z_s)Z_s^{-1}ds + \int_0^{t_{i-1}}f(s,Z_s)Z_s^{-1}ds\)\\
  &  ~~~~~~~~  +  \frac{\sigma}{4}\sum_{i = 1}^N \( W_{t_i}^H + W_{t_{i-1}}^H\) \int_{t_{i-1}}^{t_i}f(s,Z_s)Z_s^{-1}ds \\
  &  ~~~~~~~~  +  \sigma Z_0 \sum_{i = 1}^N \( W_{t_i}^H - W_{t_{i-1}}^H\)\\
  &  ~~~~~~~~  +  \frac{\sigma}{4}\sum_{i = 1}^N \(\int_0^{t_{i}}f(s,Z_s)Z_s^{-1}ds + \int_0^{t_{i-1}}f(s,Z_s)Z_s^{-1}ds\) \( W_{t_i}^H - W_{t_{i-1}}^H\)              \\
  &  ~~~~~~~~  +  \frac{\sigma^2}{4}\sum_{i = 1}^N \( W_{t_i}^H + W_{t_{i-1}}^H\)\( W_{t_i}^H - W_{t_{i-1}}^H\).
\end{aligned}
\end{equation*}

\noindent
Let
$$
X_t = X_0 + \sum_{k = 1}^{6} \mathcal{I}_k(N,t,Z_t)
$$
where
$$
\begin{cases}
  \mathcal{I}_1(N,t,Z_t) =   Z_0 \mathlarger\sum_{i = 1}^N \mathlarger\int_{t_{i-1}}^{t_i}f(s,Z_s)Z_s^{-1}ds  \\[5mm]
  \mathcal{I}_2(N,t,Z_t) = \frac{1}{4} \mathlarger\sum_{i = 1}^N \mathlarger\int_{t_{i-1}}^{t_i}f(s,Z_s)Z_s^{-1}ds\(\mathlarger\int_0^{t_{i}}f(s,Z_s)Z_s^{-1}ds + \mathlarger\int_0^{t_{i-1}}f(s,Z_s)Z_s^{-1}ds\) \\[5mm]
  \mathcal{I}_3(N,t,Z_t) = \frac{\sigma}{4}\mathlarger\sum_{i = 1}^N \( W_{t_i}^H + W_{t_{i-1}}^H\) \mathlarger\int_{t_{i-1}}^{t_i}f(s,Z_s)Z_s^{-1}ds
  \end{cases}
$$
and
$$
\begin{cases}
  \mathcal{I}_4(N,t,Z_t) = \sigma Z_0 \mathlarger\sum_{i = 1}^N \( W_{t_i}^H - W_{t_{i-1}}^H\)  \\[5mm]
  \mathcal{I}_5(N,t,Z_t) =  \frac{\sigma}{4}\mathlarger\sum_{i = 1}^N \(\mathlarger\int_0^{t_{i}}f(s,Z_s)Z_s^{-1}ds + \mathlarger\int_0^{t_{i-1}}f(s,Z_s)Z_s^{-1}ds\) \( W_{t_i}^H - W_{t_{i-1}}^H\)    \\[5mm]
  \mathcal{I}_6(N,t,Z_t) = \frac{\sigma^2}{4}\mathlarger\sum_{i = 1}^N \( W_{t_i}^H + W_{t_{i-1}}^H\)\( W_{t_i}^H - W_{t_{i-1}}^H\).
\end{cases}
$$

\noindent
Set
$$
I(t) = \int_0^{t}f(s,Z_s)Z_s^{-1}ds.
$$

\noindent
Then it follows that
\begin{equation*}
\begin{aligned}
\sum_{k = 1}^{3} \mathcal{I}_k(N,t,Z_t) &=\sum_{i = 1}^{N}\Big(I(t_i) - I(t_{i-1})\Big)Z_0 \\
                                        & ~~~~~~+ \Big(I(t_i) - I(t_{i-1})\Big)\Big(\frac{(I(t_i) + I(t_{i-1})}{4} + \frac{\sigma(W^H_{t_{i}} + W^H_{t_{i-1}})}{4} \Big).
    \end{aligned}
\end{equation*}

\noindent
Then
$$
\lim_{N\to \infty} \sum_{k = 1}^{3} \mathcal{I}_k(N,t,Z_t) = Z_0 I(t) + \frac{1}{2} \int_{0}^{t} \Big( I(s) + \sigma W^H_{s}\Big)\circ dI(s).
$$

\noindent
Since $I(s)$ is differentiable, then it follows that

\begin{equation*}
\begin{aligned}
\lim_{N\to \infty} \sum_{k = 1}^{3} \mathcal{I}_k(N,t,Z_t) & = Z_0 I(t) + \frac{1}{2} \int_{0}^{t} \Big( I(s) + \sigma W^H_{t_s}\Big) dI(s)\\
                                                            & = \Bigg(\int_0^{t}f(s,Z_s)Z_s^{-1}ds\Bigg)Z_0 \\
                                                            & ~~~~ + \frac{1}{2} \int_{0}^{t} \Bigg( \Big(\int_{0}^{s}f(u,Z_u)Z_u^{-1}du\Big) + \sigma W^H_{s}\Bigg) f(s,Z_s)Z_s^{-1}ds\\
                                                            & = \int_0^{t}f(s,Z_s)Z_s^{-1} \Bigg( Z_0 + \frac{1}{2}\int_0^{s}f(u,Z_u)Z_u^{-1}du  + \frac{\sigma}{2}W^H_{s}\Bigg)ds \\
                                                            & = \int_0^{t}f(s,Z_s)Z_s^{-1}Z_s ds = \int_0^{t}f(s,Z_s)ds. \\
    \end{aligned}
\end{equation*}

\noindent
On the other hand
\begin{equation*}
\begin{aligned}
\sum_{k = 4}^{6} \mathcal{I}_k(N,t,Z_t) & = \sigma Z_0 \mathlarger\sum_{i = 1}^N \( W_{t_i}^H - W_{t_{i-1}}^H\) \\
                                        & ~~~~ + \mathlarger\sum_{i = 1}^N \frac{\sigma}{2} \frac{I(t_i)+I(t_{i-1})}{2} \( W_{t_i}^H - W_{t_{i-1}}^H\) \\
                                        & ~~~~ + \frac{\sigma^2}{4}\mathlarger\sum_{i = 1}^N \( W_{t_i}^H + W_{t_{i-1}}^H\)\( W_{t_i}^H - W_{t_{i-1}}^H\).
\end{aligned}
\end{equation*}
Therefore,
\begin{equation*}
\begin{aligned}
\lim_{N\to \infty}\sum_{k = 4}^{6} \mathcal{I}_k(N,t,Z_t) & = \sigma Z_0 W_{t}^H + \frac{\sigma}{2} \int_{0}^{t} I(s)\circ dW_{s}^H + \frac{\sigma^2}{2}\int_{0}^{t} W_{s}^H \circ dW_{s}^H   \\
                                        & = \sigma Z_0 W_{t}^H +  \frac{\sigma}{2} \int_{0}^{t} \Big(\int_0^{s}f(u,Z_u)Z_u^{-1}du \Big) \circ dW_{s}^H + \frac{\sigma^2}{2}\int_{0}^{t} W_{s}^H \circ dW_{s}^H \\
                                        & = \sigma Z_0 W_{t}^H +  \frac{\sigma}{2} \int_{0}^{t} \Big( 2Z_s - 2Z_0 - \sigma W_{s}^H \Big) \circ dW_{s}^H + \frac{\sigma^2}{2}\int_{0}^{t} W_{s}^H \circ dW_{s}^H \\
                                        & = \sigma \int_{0}^{t} Z_s \circ dW_{s}^H.
\end{aligned}
\end{equation*}

\noindent
The third equality follows the fact that
$$\int_0^s f(u, Z_u) Z_u^{-1} du = 2 Z_s - 2Z_0 - \sigma W_s^H$$ because
           $$Z_s = Z_0 + \frac{1}{2} \int_0^s f(u, Z_u) Z_u^{-1} du + \frac{\sigma}{2} W_s^H.$$
 Now taking $N \to \infty $, that is, $\delta t \to 0$, yields
\begin{equation*}
\begin{aligned}
\lim_{N\to \infty} X_{\delta t  N} = \lim_{\delta t \to 0} X_{\delta t  N}  & = X_0 + \lim_{N\to \infty} \sum_{k = 1}^{6} \mathcal{I}_k(N,t,Z_t) \\
                       & = X_0 + \int_{0}^{t}f(s,Z_s)ds + \sigma\int_0^t Z_s\circ dW_s^H \\
                       & = X_0 + \int_{0}^{t}f\big(s,\sqrt{X_s}\big)ds + \sigma\int_0^t\sqrt{X_s}\circ dW_s^H.
    \end{aligned}
\end{equation*}

\noindent
It follows that $dX_t = f(t,\sqrt{X_t})dt + \sigma\sqrt{X_t}\circ dW_t^H $, which concludes the proof.                        \hfill{$\Box$}\\


\section{Analysis of positiveness of $(X_t)_{t\geq 0}$ for $H >  1/2$}

\begin{thm}\label{Thm3-3}
Assume that $H >\frac{1}{2}$. Let $f:[0, \infty)\times[0, \infty) \to \mathbb{R}$ be a continuous function satisfying  conditions $(D1)$ and $(D2)$.
Then the process $(Z_t)_{t\geq 0}$  defined by
    \begin{equation}\label{3-6a}
    dZ_t = \frac{f(t,Z_t)}{2Z_t}dt+\frac{\sigma}{2}dW_t^H,\,\,\, Z_0 > 0,
    \end{equation}
is strictly positive everywhere almost surely.
\end{thm}

\noindent
In the proof we shall make use of the following H\"{o}lder continuous property of fractional Brownian motion of index $H$. In the probability space $(\Omega, \mathcal{F}, \mathds{P})$, ~ $\exists\Omega'\subset \Omega, ~ \mathds{P}({\Omega'}) = 1$, such that $\forall \omega \in \Omega',$ \\ $\forall 0\leq s \leq t$ and $ \forall\alpha > 0, \,\, \exists c = c(\omega, \alpha):$
\begin{equation}\label{2-1}
\big| W^H_t(\omega) - W^H_s(\omega) \big| \leq c\big| t-s\big|^{H-\alpha}.
\end{equation}
For more background on \emph{fBm}, we refer the reader to \citet{alos2003stochastic} and \citet{nourdin2012selected}.

\begin{proof}
We have proven that condition (D1) guarantees the existence, uniqueness and positiveness of a solution up to the first time it hits zero. We shall now prove that the mere condition (D2) that the function $f(t,x) >0 $ on $[0, T] \times (0, x_T]$ for any $T>0$ and $x_T$ depending on $T$ implies that the process $(Z_t)_{t\geq 0}$ never hits zero almost surely. We shall indeded prove that $$\mathds{P}\{\omega \in \Omega: \tau(\omega) = \infty\} = 1, \mbox{where  } \tau(\omega) = \inf\{t\geq 0 : Z_t(\omega) = 0\}.$$
Let us assume that $$\mathds{P}\{\omega \in \Omega:  \tau(\omega) = \infty\} < 1 \mbox{ or equivalently }                 \mathds{P}\{\tau < T\} > 0,$$ for some fixed real $T >0$,
\noindent
and prove that this leads to a contradiction. From now on we fix a real number $x_T$ depending on $T$ for which condition (D2) holds.
Since the sample paths of \emph{fBm} $(W^H_t)_{t\geq 0}$ are (almost surely) locally H\"{o}lder continuous of order $H-\alpha$ (for each small number $\alpha >0$), then we can fix as in \citep{mishura2018fractional} a subset $\Omega_1$ of the underlying sample space $\Omega$ with $\mathds{P}(\Omega_1) = 1$ such that for each $\omega \in \Omega_1$, $\alpha > 0$,
           $$|W_t^H(\omega) - W_s^H(\omega)| \leq c |t - s|^{H-\alpha},\,\,\, \forall s, t \in [0, T]$$
\noindent
where $c = c(T, \omega, \alpha)$ is a random constant depending on $T$, $\omega$ and $\alpha$.  Our assumption $\mathds{P}(\tau < T) > 0$  implies
  $$ \mathds{P}(\tau < T) =  \mathds{P}\{\omega \in \Omega_1: \tau(\omega) < T\} > 0.$$
\noindent
Now choose $\omega \in \Omega_1$ with  $\tau(\omega) < T$.  It is given that  the process $(Z_t)$ starts at the point $Z_0 >0$. Using condition (D2), for fixed $T>0$, we take a point $x_T$ small enough such that $0 < x_T < Z_0$.
Let $S$ be the first time $(Z_t)$ hits the value $x_T$, that is,
$S = \inf\{t: Z_t = x_T\}.$  Consider a small number $\varepsilon $ such that
$0  < \varepsilon < x_T.$
Since $f(t, x) > 0$ for all $0 < t \leq T$  and $0 \leq x \leq x_T$, then in particular that
 $f(t, x) >  0$  for all $S \leq t \leq T$ and $0 \leq x \leq \varepsilon$.
Let $A = \inf\{f(t,x): S \leq t \leq T, 0 \leq x \leq x_T\}.$ Clearly $A>0$.
Let $\tau_\varepsilon$ be the {\it last time}  the process $(Z_t)$ hits $\varepsilon$ before reaching zero, that is, $$ \tau_\varepsilon(\omega) = \sup\{t\in(0,\tau(\omega)): Z_t(\omega) = \varepsilon\}.$$
\noindent
Clearly $0 < S < \tau_{\epsilon} < \tau < T.$
The equality
      $$  Z_t = Z_0 + \frac{1}{2}\int_{0}^{t} f(s,Z_s)Z_s^{-1}ds + \frac{\sigma}{2} W_t^H,  $$
\noindent
implies in particular that
      $$
      Z_\tau - Z_{\tau_\varepsilon} = \frac{1}{2}\int_{\tau_\varepsilon}^{\tau} f(s,Z_s)Z_s^{-1} ds + \frac{\sigma}{2}\left(W^H_\tau - W_{\tau_\epsilon}^H\right).
      $$
\noindent
Since $Z_{\tau} = 0$ and   $Z_{\tau_\varepsilon} = \varepsilon$, then
        $$
            \frac{1}{2}\int_{\tau_\varepsilon}^{\tau} f(s,Z_s)Z_s^{-1} ds + \frac{\sigma}{2}\left(W^H_\tau - W_{\tau_\epsilon}^H\right) = -\varepsilon
        $$
\noindent
or equivalently,
        $$
             \frac{\sigma}{2}\left(W^H_\tau - W_{\tau_\epsilon}^H\right) = -\varepsilon - \frac{1}{2}\int_{\tau_\varepsilon}^{\tau} f(s,Z_s)Z_s^{-1} ds.
        $$
\noindent
Since for all $s \in [\tau_{\varepsilon}, \tau) \subset [S, T]$, it is the case that $Z_s \in [0, \varepsilon] \subset [0, x_T]$, then by condition (D2),
       $$f(s, Z_s) > 0 \mbox{ for all } s \in [\tau_{\varepsilon}, \tau].$$
\noindent
This implies that
    $$
    \frac{\sigma}{2}\left|W^H_\tau - W_{\tau_\epsilon}^H\right|  =  \varepsilon + \frac{1}{2}\int_{\tau_\varepsilon}^{\tau} f(s,Z_s)Z_s^{-1}ds
    $$
\noindent
or equivalently
     $$
     \sigma \left|W^H_\tau - W_{\tau_\epsilon}^H\right|  =  2\varepsilon + \int_{\tau_\varepsilon}^{\tau} f(s,Z_s)Z_s^{-1}ds.
     $$
\noindent
 Since $\omega \in \Omega_1$, and $\tau_\varepsilon, \tau  \in [0, T]$, then
                   $$\left|W^H_\tau - W_{\tau_\epsilon}^H\right|  < c  \big|\tau - \tau_\varepsilon\big|^{H-\alpha}.$$
\noindent
Hence
  $$ 2\varepsilon + \int_{\tau_\varepsilon}^{\tau} f(s,Z_s)Z_s^{-1}ds \leq   \sigma c \big|\tau - \tau_\varepsilon\big|^{H-\alpha}.$$
On the other hand
\begin{equation}\label{3-7}
  \int_{\tau_\varepsilon}^{\tau} f(s,Z_s)Z_s^{-1}ds  \geq  \int_{\tau_\varepsilon}^{\tau}  A \varepsilon^{-1} ds = A \varepsilon^{-1} (\tau - \tau_\varepsilon).
\end{equation}

\noindent
Therefore
$$
2 \varepsilon + A\varepsilon^{-1} (\tau - \tau_\varepsilon) \leq \sigma c \big|\tau - \tau_\varepsilon\big|^{H-\alpha}
$$
from which it follows that
       \begin{equation}\label{3-10}
            A\varepsilon^{-1}(\tau - \tau_\varepsilon) - c\sigma\big|\tau - \tau_\varepsilon\big|^{H-\alpha} + 2\varepsilon \leq 0.
        \end{equation}
Consider the function $F_\varepsilon$ defined by
 $$F_\varepsilon (x) = A \varepsilon^{-1}x - c\sigma x^{H-\alpha} + 2\varepsilon,$$
that is, $F_\varepsilon(x)$ is obtained by replacing $\tau - \tau_\varepsilon $ with $x$. Then the inequality (\ref{3-10}) yields
\begin{equation}\label{3-11}
F_\varepsilon(\tau - \tau_\epsilon) \leq 0,
\end{equation}
\noindent
for every $\epsilon >0$. The next step in this proof is to show that the inequality in (\ref{3-11}) does not hold.
In fact we shall construct a number $\epsilon^* >0$  such that uniformly for all $0 < \varepsilon < \varepsilon^*$,
$F_\varepsilon(x) > 0$ for all $x \geq 0$. This will conclude the proof of the theorem. We will see that the conditions $H > 1/2$ and $A >0$ (based on (D2)) are necessary.
First of all, it is clear that $F_\varepsilon (0) = 2 \varepsilon  > 0$.
Let us find all critical points of $F_\varepsilon (x)$. Clearly, the first and second derivatives with respect to $x$ are respectively given by
$$
F'_\varepsilon (x) = A\varepsilon^{-1} - c\sigma(H-\alpha) x^{H-\alpha-1}
$$
\noindent
and
$$
F''_\varepsilon (x)  = - c \sigma(H-\alpha)(H-\alpha-1)x^{H-\alpha-2}.
$$
\noindent
It is clear that $F_\varepsilon (x)$ is convex as $F_\varepsilon'' (x)>0.$ Moreover, the critical point $\hat{x}$  of $F_\varepsilon(x)$ is given by
$$
\hat{x} = \(\frac{A \varepsilon^{-1}}{c\sigma(H-\alpha)}\)^{\frac{1}{H-\alpha - 1}}.
$$
Note that $\hat{x}$ is well defined since $A > 0.$ Hence,

\begin{equation*}
\begin{aligned}
F_\varepsilon (\hat{x}) &= A\varepsilon^{-1}\hat{x} - c\sigma \hat{x}^{H-\alpha} + 2\varepsilon \\
                         &= \hat{x}\(A\varepsilon^{-1} - c\sigma \hat{x}^{H-\alpha-1}\) + 2\varepsilon\\
                         &= \hat{x}\(A\varepsilon^{-1} - \frac{A \varepsilon^{-1}}{H-\alpha}\) + 2\varepsilon\\
                         &=  \frac{\hat{x}A \varepsilon^{-1}(H-\alpha-1)}{H-\alpha}+2\varepsilon\\
                         &=  \(\frac{A^{H-\alpha}}{c\sigma(H-\alpha)^{2+\alpha-H}}\)^{\frac{1}{H-\alpha - 1}}\varepsilon^{\frac{H-\alpha}{1-H+\alpha}}(H-\alpha -1) +2\varepsilon. \\
\end{aligned}
\end{equation*}
Since $H-\alpha - 1 < 0$, then

$$
F_\varepsilon (\hat{x}) \geq  \(\frac{A^{H-\alpha}}{c\sigma(H-\alpha)^{2+\alpha-H}}\)^{\frac{1}{H-\alpha - 1}}\varepsilon^{\frac{H-\alpha}{1-H+\alpha}}(H-\alpha -1) +2\varepsilon.
$$
\noindent
Set
\begin{eqnarray*}
\kappa  & = & -\(\frac{A^{H-\alpha}}{c\sigma(H-\alpha)^{2+\alpha-H}}\)^{\frac{1}{H-\alpha - 1}}(H-\alpha -1)\\
q &  = & \frac{H-\alpha}{1-H+\alpha}.
\end{eqnarray*}
\noindent
Clearly, since $H > 1/2$, we can choose $\alpha$ so small that $H > \frac{1}{2} + \alpha$ and obtain that $q \geq 1$.  Then it follows that
$$ F_\varepsilon (\hat{x})\geq - \kappa \varepsilon^{q} +2\varepsilon. $$
It is now an easy matter to show that there exists $\varepsilon^* >0$ such that for all $0 < \varepsilon < \varepsilon^*$, it is the case that
          $$ F_\varepsilon (\hat{x})\geq - \kappa \varepsilon^{q} +2\varepsilon > 0.$$
Indeed, choosing $\varepsilon^*  \leq  \left(\frac{2}{\kappa}\right)^{\frac{1}{q-1}}$ yields $ F_\varepsilon (\hat{x}) > 0$. (Note that $\varepsilon^*$ because $A \ne 0$.)  Hence $ F_\varepsilon (x) > 0$ for all $x\geq 0$. This concludes the proof of the theorem.
\end{proof}

\section{Analysis of positiveness of $(X_t)_{t\geq 0}$ for $H < 1/2$}

\noindent
We shall consider a sequence of continuous functions $$f_k(t, z): [0, \infty) \times [0, \infty) \to (-\infty, +\infty),\,\,\,k \in \mathbb{N}$$ such that each function $f_k$ satisfies conditions (D1) and (D2).
Moreover for each point $(t,z) \in [0, \infty) \times [0, \infty)$,
$$f_k(t,z) \leq f_{k+1}(t,z) \mbox{ and } \lim_{k \to \infty} f_k(t,z) = \infty.$$
Consider, for each $k$, the stochastic process $(Z_t^{(k)})_{t\geq 0}$ defined by
\begin{eqnarray*}
 Z_t^{(k)}= \left\{ \begin{array}{ll}
          Z_0 + \mathlarger\int_{0}^{t}\frac{1}{2} f_k(t,Z_s^{(k)})\(Z_s^{(k)}\)^{-1}ds+\frac{\sigma}{2} W^H_t & \mbox{ if }  t<\tau^{(k)}(\omega)\\
           0 & \mbox{ otherwise,}
           \end{array}
           \right.
 \end{eqnarray*}
where $\tau^{(k)}(\omega) = \inf\{t\geq 0 : Z_t^{(k)}(\omega) = 0\}.$
We have the following result:

\begin{thm}\label{Thm3-4}
For any $T>0$,
          $$\mathds{P}(\tau^{(k)}(\omega) > T) \to 1 \mbox{ as } k \to \infty.$$
\end{thm}
\noindent
 \textbf{Proof.}  The case where $f_k(t,z) = k - a z^2$ for some $k, a>0$ is studied by \citet{mishura2018fractional}. Their proof is based on the observation that for $k_1< k_2$,
 $$\tau^{(k_1)}(\omega) \leq \tau^{(k_2)} (\omega) \mbox{ and } Z_t^{(k_1)} (\omega) <  Z_t^{(k_2)}(\omega)$$
 for all $t$ such that $0 < t < \tau^{(k_2)}(\omega).$ It is easy to see that this extends immediately to our general case. We assume that there exist  $T>0$, an increasing sequence $(k_n)_{n>1}$ and $p>0$ such that
\begin{equation}\label{3-14}
\mathds{P}(\tau^{(k_n)} \leq T) \rightarrow p, ~~~ k_n \to \infty.
\end{equation}
\noindent
As in the previous proof, for fixed $T>0$, consider a point $x_T$ small enough such that $0 < x_T < Z_0$ and
 $S >0$ be the first time $(Z_t)$ hits the value $x_T$. Take $0  < \varepsilon < x_T.$
Then uniformly for all $k \in \mathbb{N}$,
 $f_k(t, x) >  0$  for all $S \leq t \leq T$ and $0 \leq x \leq \varepsilon$.
Let $A = \inf\{f(t,x): S \leq t \leq T, 0 \leq x \leq x_T, \}.$ Clearly $A>0$.
Also let $ \tau_\varepsilon^{(k_n)} = \sup\{t\in(0,\tau):Z_t^{(k_n)} = \varepsilon\} $ be the last hitting time of $\varepsilon$ before reaching zero.
Let
   $$A_k = \inf\{f_k(t, z): S\leq t \leq T, 0\leq z \leq Z_0\}, \,\,k >0.$$
Moreover, consider for a small number $\alpha>0$, (by the  H\"{o}lder continuity) the subspace $\Omega_1$ of probability $1$ such that
           $$|W_t^H(\omega) - W_s^H(\omega)| \leq c |t - s|^{H-\alpha}, \mbox{ for all } s, t \in [0, T]$$
where $c = c(T, \omega, \alpha)$ is a constant depending on $T$, $\omega$ and $\alpha$. Let
\begin{equation}\label{3-15}
\Omega^{(k_n)}_{T} =\big\{\omega\in\Omega_1:\tau^{(k_n)}\leq T\big\}.
\end{equation}

\noindent
Then, for all $\omega \in \Omega^{(k_n)}_{T} $,  similar arguments as in the proof of Theorem \ref{Thm3-3} yield
$$
Z_{\tau^{(k_n)}}^{(k_n)} - Z_{{\tau_\varepsilon}^{(k_n)}}^{(k_n)} = -\varepsilon = \frac{1}{2}\int_{{\tau_\varepsilon}^{(k_n)}}^{{\tau}^{(k_n)}}f_{k_n}(t,Z_s^{(k_n)})\left(Z_s^{(k_n)}\right)^{-1}ds+\frac{\sigma}{2} \big(W^H_{\tau^{(k_n)}} - W^H_{{\tau_\varepsilon}^{(k_n)}}\big).$$

\noindent
In a similar way as in the previous proof,
    $$f_{k_n}(t,Z_s^{(k_n)})\left(Z_s^{(k_n)}\right)^{-1} \geq A_{k_n} \varepsilon^{-1}, ~~~~ \forall s  \in [{\tau_\varepsilon}^{(k_n)}, {\tau}^{(k_n)}].$$
  Since
$$
\Big| W^H_{\tau^{(k_n)}} - W^H_{\tau_\varepsilon^{(k_n)}}\Big| \leq c\Big|\tau^{(k_n)} - \tau_\varepsilon^{(k_n)}\Big|^{H-\alpha},
$$
it follows (as in the previous proof) that
$$
c\sigma\Big(\tau^{(k_n)} - \tau_\varepsilon^{(k_n)}\Big)^{H-\alpha}\geq A_{k_n} \varepsilon^{-1} (\tau^{(k_n)} - \tau_\varepsilon^{(k_n)}) +2\varepsilon.
$$

\noindent
 This implies in particular that
\begin{equation}\label{3-16}
\begin{cases}
  c\sigma\Big(\tau^{(k_n)} - \tau_\varepsilon^{(k_n)}\Big)^{H-\alpha}\geq  2\varepsilon \\
  c\sigma\Big(\tau^{(k_n)} - \tau_\varepsilon^{(k_n)}\Big)^{H-\alpha}\geq A_{k_n}(\tau^{(k_n)} - \tau_\varepsilon^{(k_n)}) \varepsilon^{-1}.
\end{cases}
\end{equation}

 \noindent
We shall show that the two inequalities are contradictory. Elementary calculations show that the second inequality in (\ref{3-16}) is equivalent to
   $$\Big(\tau^{(k_n)} - \tau_\varepsilon^{(k_n)}\Big) \leq \left(\frac{1}{c \sigma} A_{k_n} \varepsilon^{-1} \right)^{\frac{1}{H-\alpha -1}}$$
\noindent
Taking both sides with power $H - \alpha$ and thereafter multiplying both sides by $c \sigma$ yields
\begin{eqnarray*}
c\sigma\Big(\tau^{(k_n)} - \tau_\varepsilon^{(k_n)}\Big)^{H-\alpha} & \leq &   c \sigma \left(\frac{1}{c \sigma} A_{k_n} \varepsilon^{-1} \right)^{\frac{H-\alpha}{H-\alpha -1}}\\
    & = & \left(c^{\frac{1}{1-H+\alpha}}\right) \left( \sigma^{\frac{1}{1-H+\alpha}}\right) \varepsilon^{\frac{H-\alpha}{1-H+\alpha}} \left(A_{k_n}\right)^{-\frac{H-\alpha}{1- H+\alpha}}
\end{eqnarray*}
In the right hand side, the H\"older constant $c=c(\omega)$ is random depending on the path $\omega$ of \emph{fBm}. As in \citet{mishura2018fractional}, it is well-known that $c(\omega)$ is finite almost surely and hence since $\mathds{P}\Big(\bigcap_{n>1} \Omega^{(k_n)}_{T} \Big) = p  >0$, then there exists a (non-random) constant $M$ and a subset $E$ of $\bigcap_{n>1} \Omega^{(k_n)}_{T}$ with $\mathds{P}(E) >0$ such that
      $c = c(\omega) \leq M$ for all $\omega \in E$.
Hence, everywhere in $E$,
      \begin{eqnarray*}
c\sigma\Big(\tau^{(k_n)} - \tau_\varepsilon^{(k_n)}\Big)^{H-\alpha} \leq  \left(M ^{\frac{1}{1-H+\alpha}}\right) \left( \sigma^{\frac{1}{1-H+\alpha}}\right) \varepsilon^{\frac{H-\alpha}{1-H+\alpha}} \big(A_{k_n}\big)^{-\frac{H-\alpha}{1-H+\alpha}}.
\end{eqnarray*}
Clearly $M$ and $\sigma$ are constants. Moreover, since $f_n(t,z) \to \infty$ as $n \to \infty$ (for every $(t,z)$) then clearly also $A_{k_n} \to \infty$ for $k_n \to \infty$.  Hence
           $$\lim_{k_n\to \infty}  \left(A_{k_n}\right)^{-\frac{H-\alpha}{1-H+\alpha}} = 0, $$
because $-\frac{H-\alpha}{1-H+\alpha} < 0.$ Then clearly, for any given $\varepsilon >0$,   we can choose $k_n$ very large (depending on $\varepsilon$) such that
                       $$\left(M ^{\frac{1}{1-H+\alpha}}\right) \left( \sigma^{\frac{1}{1-H+\alpha}}\right) \varepsilon^{\frac{H-\alpha}{1-H+\alpha}} \left(A_{k_n} \right)^{-\frac{H-\alpha}{1-H+\alpha}} < 2 \varepsilon.$$
This yields
          $$ c\sigma\Big(\tau^{(k_n)} - \tau_\varepsilon^{(k_n)}\Big)^{H-\alpha} <  2\varepsilon, $$ which contradicts the first inequality in (\ref{3-16}). This concludes the proof of the theorem. \hfill{$\Box$}

\section{Some illustrating examples with simulations}\label{Illust}
In this section, we provide some examples of generalised \emph{fCIR} processes to illustrate the results of this paper using simulations. The process that will be used represents a generalisation of the classical \emph{``extended CIR''} process. \\

\noindent
The classical extended \emph{CIR} process is defined by
\begin{equation}\label{4-1}
  dX_t = \theta_t(\mu_t - X_t)dt + \sigma\sqrt{X_t}dW_t, \,\, X_0 > 0
\end{equation}
\noindent
where $\theta_t$ is the time-depending speed of reversion towards its time-depending long run mean $\mu_t$ of the process $(X_t)_{t\geq 0}$ and $\sigma$ a positive parameter. This model was initially introduced by \citet{hull1990pricing} and it is widely used in both short interest rates and spot volatilities modelling. The choice of parameters $\theta_t$ and $\mu_t$ are done through market calibration.
As already discussed,  we shall consider the general case where the Brownian motion is replaced with a \emph{fBm}. The process is called \emph{``Extended fCIR''} and takes the form
  \begin{equation}\label{4-1a}
  X_t = Z_t^2 \mathbf{1}_{[0, \tau)}, ~~~ t\geq 0
  \end{equation}
where
    \begin{equation}\label{4-1b}
        dZ_t = \frac{f(t,Z_t)}{2Z_t}dt+\frac{\sigma}{2}dW_t^H, ~~~ Z_0 > 0
    \end{equation}
with
  \begin{equation}\label{4-1c}
 f(t, x) = \theta_t(\mu_t - x^2).
 \end{equation}

\noindent
We shall then simulate the  corresponding process $(X_t)$ on a finite interval $[0, T]$ using the well-known Euler-Maruyama method. (See e.g. \citet{higham2002strong} for more details about the method.)  Subdivide the interval $[0,T]$ into $N$ subintervals of equal length $\delta t = T/N$ with end points $0=t_0, t_1, t_2, \ldots,t_N = T$. The corresponding discrete  version of the process $(X_t)_{t\geq 0}$ is given by
$$
X_{t_n} = Z_{t_n}^2,
$$
\noindent
where $Z_0 >0$ and
 for $n = 1,2,\ldots, N$,
\begin{eqnarray*}
Z_{t_n} = \left\{ \begin{array}{cc}
Z_{n-1} + \frac{f(t_{n-1},Z_{t_{n-1}})}{2Z_{t_{n-1}}}\delta t + \frac{\sigma}{2} \delta W^H_{t_n} & \mbox{ if } Z_{t_{n-1}} > 0,\\
0 & \mbox{ otherwise}
\end{array}
\right.
\end{eqnarray*}
with
$$\delta W^H_{t_n} = W^H_{t_n} - W^H_{t_{n-1}}.$$

\noindent
In what follows, we shall consider two different drift functions for simulation of the process (\ref{4-1a}).

\subsection*{Illustration I}

We consider $\theta_t = \theta >0$ and

  $$
  \mu_t = c + \frac{\sigma^2}{2\theta} \Big(1 - e^{-2\theta t}\Big)
  $$
\noindent
where $c>0$ is a constant. This yields the drift function
   \begin{equation}\label{4-3}
  f(t,x) =  \frac{\sigma^2}{2} \Big(1 - e^{-2\theta t}\Big) + \theta(c - x^2), \,\,\, t\geq 0, x \geq 0.
\end{equation}
\noindent
It is clear that the function $f(t,x)$ satisfies conditions (D1) and (D2) and hence for
We simulate 1000 sample paths of the process $(X_t)_{t\in[0,T]}$  where $T = 10$, volatility $\sigma = 0.4$ starting at $X_0 = 1$ with time-step $\delta t = 0.001$ and the results are given in Figures 4.1 to 4.4 (with given parameters $c$, $\theta$ and $H$).
All the sample paths in Figures 4.1 and 4.2 where $H >0.5$ are strictly positive (do not hit zero) in line with Theorem 3.3.


\begin{center}
\begin{tabular}{cc}
  \includegraphics[width=2.4in, height=2.2in]{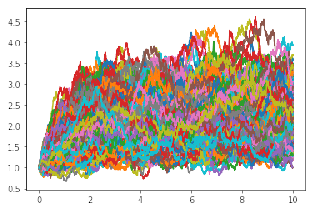} & \includegraphics[width=2.5in, height=2.2in]{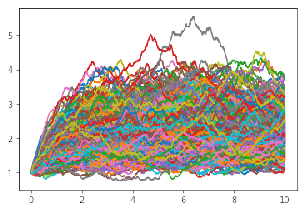}  \\
  Figure 4.1 & Figure 4.2 \\
  $\theta = 1, ~ c = 2, ~H=0.6$ & $\theta = 1, ~ c = 2, ~H=0.8$ \\
  \end{tabular}
\end{center}






\subsection*{Illustration II}

In the second illustration, we consider again $\theta_t = \theta >0$, $\sigma > 0$ and

$$
\mu_t = \Big(1 + \frac{c}{\theta}\Big )e^{ct} + \frac{\sigma^2}{2\theta} \Big(1 - e^{-2\theta t}\Big),
$$

\noindent
where $c>0$ is a constant. This yields the function

 \begin{equation}\label{4-6}
  f(t,x) = \Big(\theta + c\Big )e^{ct} + \frac{\sigma^2}{2} \Big(1 - e^{-2\theta t}\Big) - \theta x^2,
\end{equation}
It is again clear that $f(t,x)$ satisfies conditions (D1) and (D2). We considered  1 000 realisations of the sample paths of the stochastic process $(X_t)_{t\in[0,10]}$ with volatility $\sigma = 0.4$ starting at $X_0 = 1$ with time-step $\delta t = 0.001$. We have observed similar results compared to Simulation I and the output is given from Figures  4.6 to 4.9.

\begin{center}
\begin{tabular}{cc}
  \includegraphics[width=2.5in, height=2.3in]{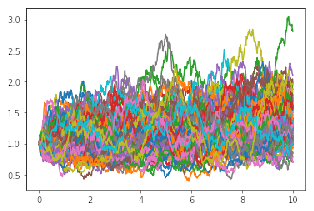} & \includegraphics[width=2.5in, height=2.3in]{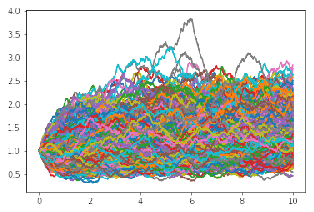}  \\
  Figure 4.6 & Figure 4.7 \\
  $\theta = 1, ~ c = 0.02, ~H=0.6$ & $\theta = 1, ~ c = 0.02, ~H=0.8$ \\
  \end{tabular}
\end{center}

\begin{center}
\begin{tabular}{cc}
 \includegraphics[width=2.5in, height=2.3in]{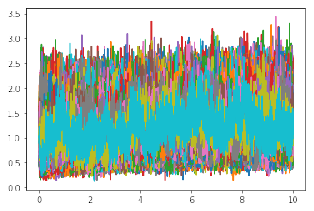}  & \includegraphics[width=2.5in, height=2.4in]{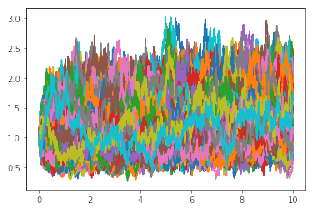}  \\
 Figure 4.8 & Figure 4.9\\
 $\theta = 1, ~ c = 0.02, ~H=0.1$ & $\theta = 1, ~ c = 0.02, ~H=0.4$ \\
\end{tabular}
\end{center}


\section*{Concluding remarks}

In this work, we analysed the general \emph{fCIR} processes of the form $X^2_t = Z^2_t \mathbf{1}_{[0, \tau)}$ with  \begin{equation*}
            dZ_t = \mfrac{1}{2}f(t,Z_t)Z_t^{-1}dt + \mfrac{1}{2}\sigma dW_t^H,\,\, Z_0 >0,
        \end{equation*} where $f(t,x)$ is a continuous function on $\mathds{R}_+^2$ under two mild conditions on the function $f(t,x)$.
We proved that the process $(X_t)$ satisfies the equation $dX_t = f(t, \sqrt{X_t}) dt + \sigma \sqrt{X_t} \circ dW_t^H$. Moreover if the Hurst parameter  $H > 1/2$, the process $(X_t)_{t\geq 0}$ processes will never hit zero, that is, it remains strictly positive everywhere almost surely. The conditions (D1) and (D2) imposed on $f(t,x)$ are very weak so that the class of functions to which our results apply is clearly larger than previously understood.
In the case, $H<1/2$, we considered a sequence of increasing drift functions $(f_n)$ that tends to infinity and we proved that the probability of hitting zero converges to zero as $n$ goes to infinity. These results are illustrated with some simulations. The generalised \emph{fCIR} process may take several forms and one of them is given as an extended \emph{fCIR} or fractional Hull-White process. This process belongs to the class of mean-reverting processes and may yield perfect calibrations of time-dependent parameters. Calibration under \emph{fCIR} process constitutes an important area of further investigations. Another line of further research is to study the properties of moments of the process $(X_t)$ in order to see if results that have been obtained under more stronger conditions remain valid under the mild conditions (D1) and (D2). We hope to the results and discussions in this paper will be of some help in that direction. It is important to note that our results generalise previous results obtained by  \citet{mishura2018fractional} to the particular function $f(t,x) = \frac{1}{2}(\mu - \theta x^2)$ for constants $\mu>0$ and $\theta >0$.

\paragraph{Acknowledgment:} We would like to thank the anonymous reviewer whose comments have greatly improved the paper.

\nocite {*}

\bibliography {FCIR_ArXiV}

\bibliographystyle {plainnat}

\end{document}